\documentclass[11pt]{amsart}

\usepackage{amsmath,amssymb,amsfonts,amsthm,mathrsfs}
\usepackage{enumerate}
\usepackage[all]{xy}
\usepackage{MnSymbol}

\usepackage[top=2.5cm,bottom=2.5cm,right=3cm,left=3cm]{geometry}

\usepackage{graphicx}
\usepackage{hyperref}

\newtheorem{theorem}{Theorem}[section]
\newtheorem{proposition}[theorem]{Proposition}
\newtheorem{lemma}[theorem]{Lemma}

\theoremstyle{definition}
\newtheorem{remark}[theorem]{Remark}
\newtheorem{definition}[theorem]{Definition}
\newtheorem{notation}[theorem]{Notation}

\newcommand{\R}{\mathbb{R}}
\newcommand{\C}{\mathbb{C}}
\newcommand{\Z}{\mathbb{Z}}
\newcommand{\Q}{\mathbb{Q}}

\newcommand\cO{\mathcal{O}}
\newcommand\gp{\mathfrak{p}}

\newcommand\sF{\mathscr{F}}
\newcommand\sG{\mathscr{G}}
\newcommand\sL{\mathscr{L}}
\newcommand\sS{\mathscr{S}}

\DeclareMathOperator{\rG}{G}
\DeclareMathOperator{\rH}{H}
\DeclareMathOperator{\rI}{I}

\DeclareMathOperator{\Gal}{Gal}
\DeclareMathOperator{\Hom}{Hom}
\DeclareMathOperator{\ad}{ad}
\DeclareMathOperator{\Reg}{Reg}
\DeclareMathOperator{\Frob}{Frob}
\DeclareMathOperator{\GL}{GL}

\newcommand{\loc}{\mathrm{loc}}
\newcommand{\Ind}{\mathrm{Ind}}
\newcommand{\ord}{\mathrm{ord}}
\newcommand{\Fil}{\mathrm{Fil}}
\newcommand{\dR}{\mathrm{dR}}
\newcommand{\crys}{\mathrm{crys}}
\newcommand{\BK}{\mathrm{BK}}
\newcommand{\sLm}{\mathscr{L}_{\mbox{-}}}

\author{Mladen Dimitrov}
\address{University of Lille, CNRS, UMR 8524 -- Laboratoire Paul Painlev\'e, 59000 Lille, France.}
\email{mladen.dimitrov@univ-lille.fr}
\author{Alexandre Maksoud}
\address{University of Luxembourg, Unit\'e de Recherche en Math\'ematiques, Maison du nombre, 6~avenue de la Fonte, 4364 Esch-sur-Alzette, Luxembourg. }
\email{alexandre.maksoud@uni.lu}

\dedicatory{To Bernadette Perrin-Riou on the occasion of her 65th birthday}

\title{$\sL$-invariants of Artin motives}

\begin{document}
\maketitle

\begin{abstract}
We compute Benois $\sL$-invariants of weight $1$ cuspforms and of their adjoint representations 
and show how this extends Gross' $p$-adic regulator to   Artin motives 
which are not critical in the sense of Deligne. Benois' construction depends on the choice of a  regular submodule
which is well understood when the representation is $p$-regular, as it then amounts to the choice of a ``motivic'' $p$-refinement. 
The situation is dramatically different in the $p$-irregular case, where the  regular submodules are 
parametrized by a flag variety and thus depend on  continuous parameters.  We are nevertheless able to show in some examples, how Hida theory and the geometry of the eigencurve can be used to detect a finite number of choices of arithmetic 
and ``mixed-motivic'' significance. 
\end{abstract}

\setcounter{tocdepth}{2}
\tableofcontents

\section{Introduction}
Let $p$ be a prime number and let $\varrho : \Gal(\overline{\Q}/\Q) \to \GL_{\overline{\Q}_p}(W)$ be a $d$-dimensional  $p$-adic Galois representation having  finite image and not containing the trivial representation, i.e.  $\rH^0(\Q,W)=0$.
 After choosing an embedding $\iota_p : \overline{\Q} \subseteq \overline{\Q}_p$ one may think of $\varrho$ as the $p$-adic realization of an Artin motive $[\varrho]$ with $\overline{\Q}$-coefficients whose arithmetic dual $[\varrho]^*(1)$ will be denoted by $M$.
When $\varrho$ is unramified at $p$, then the $p$-adic realization $V=W^*(1)$ of $M$ is crystalline at $p$ and, given {\it any} regular submodule $D\subset D_\crys(V)$ (see Def.~\ref{def:module_reg}),  Perrin-Riou \cite{perrin-riou:Lp} predicts the existence of  $p$-adic $L$-functions $L_p(V,D,s)$ and $L_p(W,D^\perp,s)$  interpolating the corrected at $p$ values of the Artin $L$-functions $L(\varrho^*(1),s)$ and $L(\varrho,s)$ at rational integers. When the correcting Euler factor at $p$  vanishes, we say that 
the $p$-adic $L$-function admits a trivial (or extra)  zero. Under some mild assumptions, which are satisfied by all $M$ as above, 
Benois has proposed in \cite[\textsection5.1.2]{benoiscrys} a formula for the leading term of the Taylor's expansion at $s=0$ of $L_p(V,D,s)$. 
This formula, called the Trivial Zero Conjecture,  differs from the interpolation formula proposed by Perrin-Riou by a certain scalar
 $\sL(V,D)$ defined in terms of Bloch--Kato Selmer groups $\rH^1_f(\Q, V)\subset \rH^1_{f,p}(\Q, V)$ (whose definitions are recalled in \textsection\ref{BK-groups}) and 
  generalizing Greenberg's $\sL$-invariant\footnote{The definition of $L_p(V,D,s)$, resp.   $L_p(W,D^\perp,s)$, depends on the choice of a Galois stable lattice in $V$, resp.  $D$, but $\sL(V,D)$, resp. $\sL(W,D^\perp)$, does not, hence our abuse of notation.}. 

Trivial zeros  occur abundantly in our setting and were  discovered by Gross when $\varrho$ is induced from a totally odd Hecke character over a totally real field. However the construction of cohomological $\sL$-invariants for a general Artin motive, presented in Section~\ref{benois-L}, was only recently introduced by Benois. 
Rewriting Benois' definition as a determinant involving explicit linear combinations (with algebraic coefficients) of $p$-adic logarithms of  $p$-units (Theorem~\ref{thm:calcul_L_inv}),  allows us to connect $\sL(V,D)$ with Gross' regulator (see \textsection\ref{sec:gross}) and Hida's $\sL$-invariant for the adjoint of a $p$-regular weight one cuspform (see \textsection\ref{sec:ad_reg}). 

Another fundamental aspect of the theory of $p$-adic $L$-functions is to determine whether the $\sL$-invariant vanishes. The Weak $p$-adic Schanuel Conjecture from transcendental number theory strongly indicates that  $\sL(V,D)\ne 0$  when the regular submodule $D$ is motivic, i.e.  has a $\overline{\Q}$-structure, which is always the case when $V$ is regular at $p$. In contrast, in the $p$-irregular setting 
$D$ depends on  continuous parameters and $\sL(V,D)$ could very well vanish for some $D$. 
 We illustrate the dependence of $\sL(V,D)$ on $D$ in the baby case of a weight one cuspform with complex multiplication  (CM)  which is irregular at $p$ (see \textsection\ref{CM-case}).

Constructing non-trivial classes in Bloch--Kato Selmer groups is a central yet notoriously difficult problem in number theory as witnesses the long-standing Birch and Swinnerton--Dyer Conjecture. The Euler Systems approach, that we will not pursue here, sometimes succeeds in writing down such classes  explicitly ({\it e.g.} by using cyclotomic, resp. elliptic, units
in the case of the Kubota--Leopoldt, resp. Katz, p-adic $L$-function). 
  While a motive $M$ as above cannot be adjoint (as it is pure of motivic weight $-2$), a theorem of Benois--Horte showing $\sL(V,D)=(-1)^e \sL(W,D^\perp)$, where  $e$ is the order of the trivial zero, allows us to use  its motivic dual $M^*(1)$ 
  which is pure of  weight $0$ and thus can be adjoint. 
 The advantage then is that one can use  Mazur's deformation theory of Galois representations to study first order deformations of the $p$-adic realization of a motive and produce classes in the $\rH^1$ of its adjoint. 
   A particularly interesting case, discussed in \textsection\ref{adjoint-CM-irreg}, is that of the adjoint of a weight $1$ CM cuspform $f$ which is irregular at $p$, for which one encounters a double trivial zero. Using some results in 
   \cite{betina-dimitrov} and the explicit description of $\sL(W,D^\perp)$ in \textsection\ref{sec:dualL}, one  deduces an intriguing expression \eqref{L-inv-F} for the $\sL$-invariant in the  direction
   of a non-CM  Hida family $\sF$ containing $f$ (see \textsection\ref{CM-case} and \textsection\ref{sec:basis} for notation):    
\[\sL(\ad^0(\varrho_f), W_{\sF}^+)=
\frac{2\sL_\gp\cdot \left(\sLm(\bar\varphi)  \sL(\varphi)+  \sLm(\varphi) \sL(\bar\varphi)\right)}{\sLm(\bar\varphi)+\sLm(\varphi)}. \]
An interesting feature of this formula is the presence of the \emph{anti-cyclotomic} $\sL$-invariants $\sLm(\varphi)$ and $\sLm(\bar\varphi)$. Although the regular submodule $D$ associated with $\sF$ has no natural $\overline{\Q}$-structure, the $\sL$-invariant is an algebraic expression in $p$-adic logarithms of $p$-units, hence it should not vanish according to the Weak $p$-adic Schanuel Conjecture.

 Let us mention that, on the algebraic side, it was first shown by Greenberg \cite{greenberg:l-adic}, Federer-Gross \cite{federer-gross} and later proved in greater generality by Benois \cite{benoiscrys} via the machinery of Selmer complexes that $\sL$-invariants control the fine structure of Iwasawa modules attached to $\varrho$. These control theorems might be taken as an additional motivation for the computation of Benois' $\sL$-invariant for arbitrary Artin motives, where a construction of a Perrin-Riou's $p$-adic $L$-function is at the moment totally out of reach.

\section{Computation of Benois $\sL$-invariants}\label{benois-L}

Given a field $F$ of characteristic $0$, we let $\cO_F$ denote its ring of integers and 
$\rG_F$ its absolute Galois group. When $F$ is local, we let 
$\rI_F\subset \rG_F$ denote the inertia subgroup. 

For any rational prime $\ell$ we choose an embedding  $\iota_\ell : \overline{\Q} \subseteq \overline{\Q}_\ell$. 
We also  fix an embedding $\iota_\infty : \overline{\Q} \subseteq \C$ which determines a  
 complex conjugation $\tau \in G_{\Q}$.

\subsection{Bloch--Kato Selmer groups}  \label{BK-groups}

\subsubsection{Local and global Selmer groups} Let $V$ be a $p$-adic representation of $\rG_F$. 

If $F$ is a finite extension of $\Q_\ell$, the  local Bloch--Kato subgroup is  defined as
\[\rH^1_f(F,V)=\begin{cases} 
	\ker\left( \rH^1(F,V)\to \rH^1(\rI_F,V)\right)& \text{ , if } \ell\ne p, \\
	\ker\left( \rH^1(F,V)\to \rH^1(F, B_\crys \otimes_{\Q_p }V)\right)& \text{ , if } \ell=p,
\end{cases}\]
where $B_\crys$ is the Fontaine ring of crystalline periods (see \cite[\textsection3]{bloch-kato}). 

If $F$ is a number field and   $S$ a   finite set  of finite places of  $F$, then the  Selmer group is defined as
\[\rH^1_S(F,V)=
\ker\left( \rH^1(F,V)\to \bigoplus_{v\notin S} 
\frac{\rH^1(F_v, V)}{\rH^1_f(F_v, V)}\right).
\]
Here, the map $\rH^1(F,V)\to \rH^1(F_v,V)$ is given by the restriction of cocycles to $\rG_{F_v}$ and shall be denoted by $\loc_v$.
We will be particularly interested in the Bloch--Kato Selmer group 
$\rH^1_f(F,V)=\rH^1_\varnothing(F,V)$ and the 
group $\rH^1_{f,p}(F,V)=\rH^1_{\{p\}}(F,V)\supset \rH^1_f(F,V)$ 
consisting of cocycles that are unramified outside $p$. 

\subsubsection{Kummer Theory}
When $V=\overline{\Q}_p(1)$, Kummer theory provides an explicit description of the Bloch--Kato Selmer groups. For all integers 
$n\geqslant 1$, the Kummer map $\kappa_{F,n}: F^\times/(F^{\times})^{p^n} \to \rH^1(F,\mu_{p^n})$ is an isomorphism sending $a\in F^\times$ to the class of the cocycle given by $g\mapsto g(a^{1/p^n})/a^{1/p^n}$, where $a^{1/p^n}\in \overline{F}^\times$ is any $p^n$-th root of $a$ (see \cite[Chap.~X, \textsection3.b)]{serre:cohomologie_galoisienne}). It is clear from the definition that $\kappa_{F,n}$  is functorial in $F$ and Galois equivariant. Taking an inverse limit on $n$ and extending scalars to  $\overline{\Q}_p$ yields a Galois equivariant isomorphism 
\[ \kappa_F : \widehat{F^\times} \otimes_{\Z_p } \overline{\Q}_p  \xrightarrow{\sim}  \rH^1(F,\overline{\Q}_p(1)),  \]
where $\widehat{A} = \varprojlim_n A / A^{p^n}$ denotes the $p$-adic completion of an abelian group $A$. Moreover, when $F$ is local, $\kappa_F$ restricts to an isomorphism
\[\kappa_F : \widehat{\cO_F^\times}  \otimes_{\Z_p }  \overline{\Q}_p \xrightarrow{\sim}  \rH^1_f(F,\overline{\Q}_p(1)), \]
(see e.g. \cite[Example 3.9]{bloch-kato}).
Note that if $F$ is a finite extension of $\Q_\ell$ and $\ell\ne p$, then $\rH^1_f(F,\overline{\Q}_p(1))$ vanishes, as   
$\widehat{\cO_F^\times}$ is finite.  In contrast, for a finite extension $F$ of $\Q_p$,  
$\widehat{\cO_F^\times}$   is  the subgroup $\cO_F^{1}$ of principal units of $\cO_F^\times$.  
Finally, when $F$ is a number field and $S$ a finite set of finite places of $F$, the global Kummer map restricts to a Galois equivariant isomorphism  
\begin{equation}\label{eq:BK-global}
\kappa_F: \widehat{\cO_F\left[\tfrac{1}{S}\right]^\times}  \otimes_{\Z_p }  \overline{\Q}_p \xrightarrow{\sim}
\rH^1_S(F,\overline{\Q}_p(1)),
\end{equation}
where $\cO_F\left[\tfrac{1}{S}\right]^\times$ denotes the group of $S$-units of $F$.

Our next task is to give an explicit description of Bloch--Kato Selmer groups for (duals of) Artin representations.  

\subsubsection{Units}

Henceforth $V=W^*(1)$, where $W$ is a $\overline{\Q}_p$-valued Artin representation unramified at $p$ as in the introduction.
Denote by  $\rH^1_f(V)\subseteq \rH^1(V)=\rH^1(\Q,V)$  the global Bloch--Kato Selmer group of $V$ and by
	 $\rH^1_{f,p}(V) $  the group consisting of cocycles that are unramified outside $p$.  
	 
Let $H$ be the fixed field of the kernel of $\varrho : \Gal(\overline{\Q}/\Q) \to \GL_{\overline{\Q}_p}(W)$ 
and $G=\Gal(H/\Q)$.   Similarly, we let  $E$ be the fixed field of $\ker(\varrho_{\vert \rG_{\Q_p }})$ and  $G_p=\Gal(E/\Q_p )$. Thus, $E$ is the $p$-adic completion of $H$ with respect to the place defined by $\iota_p$. 

\begin{definition}
We let 	$U_H = \cO_H^\times \otimes_{\Z} \overline{\Q}_p$ and $U_H^{(p)}=\cO_H\left[\frac{1}{p}\right]^\times \otimes_{\Z} \overline{\Q}_p$ and we denote by  $U_{H,\overline{\Q}}=\cO_H^\times \otimes_{\Z} \overline{\Q}$ and
  $U^{(p)}_{H,\overline{\Q}}=\cO_H\left[\frac{1}{p}\right]^\times \otimes_{\Z} \overline{\Q}$ the natural  $\overline{\Q}$-subspaces. 
\end{definition}

\begin{lemma}\label{lem:BK_desc}
 	\begin{enumerate}[(i)]
	\item The local Kummer map $\kappa_E$ yields a canonical identification: 
\[\rH^1_f(\Q_p ,V)=\Hom_{G_p}(W,\cO_E^{1}\otimes_{\Z_p } \overline{\Q}_p).\] 
	\item The global Kummer map	$\kappa_H$ yields  canonical  identifications: 
\[	\rH^1_f(V)=\Hom_G(W,U_H) \text{ and } \rH^1_{f,p}(V)=\Hom_G(W,{U_H^{(p)}}).\]  
	 \item The localization map
 \[\loc_p : \rH^1_f(V) \longrightarrow \rH^1_f(\Q_p ,V)\]
 is induced by the diagonal embedding  $U_H=\widehat{\cO_H^\times} \otimes_{\Z_p} \overline{\Q}_p \to \widehat{(\cO_H\otimes\Z_p)^\times}\otimes_{\Z_p} \overline{\Q}_p$ and  Frobenius reciprocity isomorphism 
  \[\Hom_{G}(W,\widehat{(\cO_H\otimes\Z_p)^\times}\otimes_{\Z_p} \overline{\Q}_p)=
  \Hom_{G_p}(W,\cO_E^{1}\otimes_{\Z_p } \overline{\Q}_p).\]
 \end{enumerate}
\end{lemma}
\begin{proof} By the  inflation-restriction exact sequence,  $\rH^1_f(\Q_p ,V)$ is canonically isomorphic with  
\[\rH^1_f(E,V)^{G_p}=\left(\rH^1_f(E,\overline{\Q}_p(1))\otimes W^*\right)^{G_p}=\Hom_{G_p}\left(W,\rH^1_f(E,\overline{\Q}_p(1))\right)\] 
and the first statement then follows form the fact that   $\kappa_E$ is $G_p$-equivariant. 
The second  statement is obtained similarly using in addition that $\widehat{\cO_H\left[\tfrac{1}{S}\right]^\times}=\cO_H\left[\tfrac{1}{S}\right]^\times  \otimes \Z_p $. The third  statement is a consequence  of the other two
(see also \cite[Appendix~B, Cor.~ 5.3.(ii)]{rubinES})
\end{proof}

\begin{remark}
\begin{enumerate}[(i)]
	\item By  the finiteness of the ideal class group of $H$, we have $\rH^1_f(V^*(1))=\rH^1_f(W)=\{0\}$, 
	in  concordance  with the non-vanishing of $L(\varrho,1)$. 
	 \item  Let $W_{\overline{\Q}}$ be a   $\overline{\Q}$-structure in the finite image
	 $\rG_{\Q}$-representation $W$ with respect to which the motive $M^*(1)$ is defined.  	 
	 Let  $\rH^1_f(M)$, resp.  $\rH^1_{f,p}(M)$, be the group  classifying extensions of the trivial motive by $M$
	 which are unramified everywhere, resp. everywhere outside $p$, in the category of motives over $\Q$ with $\overline{\Q}$-coefficient	 (see \textit{e.g.}  \cite[\textsection2.2]{colmezfonctions}). 
	 It is expected that 
	 \begin{align*}
	 \rH^1_f(M) & =\rH^1_f(V)_{\overline{\Q}}:= \Hom_G(W_{\bar \Q},U_{H,\overline{\Q}}),  \text{ and }   \\\rH^1_{f,p}(M) & =\rH^1_{f,p}(V)_{\overline{\Q}}:= \Hom_G(W_{\bar \Q},{U_{H,\overline{\Q}}^{(p)}}).
	 \end{align*} 
	  Note that  by Lemma~\ref{lem:BK_desc}(ii)  this is consistent with Jannsen's Conjecture (see \cite[Conj.~2.3]{colmezfonctions})  identifying $\rH^1_f(M)_{\overline{\Q}_p}$ and $\rH^1_{f,p}(M)_{\overline{\Q}_p}$ with $\rH^1_f(V)$ and $\rH^1_{f,p}(V)$, respectively. 
\end{enumerate}
\end{remark}

\subsubsection{Dimensions} \label{sec:dim}
Let $d^+=\dim \rH^0(\R,W)$, resp.  $d^-=d-d^+$,  be the dimension of the largest subspace of $W$ on which  the complex conjugation acts as $1$, resp. $-1$,  and let $f=\dim \rH^0(\Q_p ,W)$. By Dirichlet's Unit Theorem, the Bloch--Kato Selmer groups $\rH^1_f(V)$ and $\rH^1_{f,p}(V)$ have dimension $d^+$ and $d^+ + f$, respectively.

\subsubsection{Transcendence Theory}\label{subsec:log_p} 
The $\Z_p$-module $ \widehat{E^{\times}}= p^{\Z_p}\times \cO_E^1 $ is endowed with a $p$-adic valuation map 
 $\ord_p : \widehat{E^{\times}} \to \Z_p $, factoring through the  first projection, and  
 an  Iwasawa determination of a $p$-adic logarithm $\log_p : \widehat{E^{\times}} \to  p\cO_E\subset \overline{\Q}_p $, factoring through the second  projection. 
 We consider their $\overline{\Q}_p$-linear extensions: 
 \[ \ord_p : \widehat{E^{\times}} \otimes_{\Z_p}  \overline{\Q}_p \longrightarrow \overline{\Q}_p\quad \text{ and } \quad \log_p : \widehat{E^{\times}} \otimes_{\Z_p}  \overline{\Q}_p \longrightarrow \overline{\Q}_p.\]
 Pre-composing with  $\iota_p$ allows one to 
 apply both   $\log_p$ and $\ord_p$  to elements of $U_{H}^{(p)}\supset U_{H}$ and we shall do so without further notice. 
 
 Let us recall a consequence of Brumer's $p$-adic version of Baker's Theorem on the independence of logarithms of global units \cite{brumer}.

\begin{theorem}[Baker--Brumer]\label{thm:BB}
	The  map $\log_p : U_{H,\overline{\Q}} \to \overline{\Q}_p$ is injective.
\end{theorem}

By the  $G_p$-equivariance of the  $p$-adic logarithm and  Frobenius reciprocity,  we have: 
\begin{equation} \label{eq:frob-rec}
\Hom_{G_p}(W,\cO_E^{1}\otimes_{\Z_p } \overline{\Q}_p)\underset{\sim}{\xrightarrow{\log_p}}
\Hom_{G_p}(W,E\otimes_{\Q_p } \overline{\Q}_p)\xrightarrow{\sim} \Hom(W,\overline{\Q}_p).
\end{equation} 
Composing  this isomorphism with  $\loc_p$ from Lemma~\ref{lem:BK_desc}(iii) yields  a map
\[\log_p\circ  \loc_p:
\Hom_{G}(W,U_H)\to \Hom_{G_p}(W,\cO_E^{1}\otimes_{\Z_p } \overline{\Q}_p)\xrightarrow{\sim} \Hom(W,\overline{\Q}_p),\]
whose injectivity is the content of the $\varrho$-part of the Leopoldt Conjecture.  

Note that the Leopoldt Conjecture  for $H$ at $p$, 
asserting the injectivity of $\widehat{\cO_H^\times} \otimes_{\Z_p} \Q_p \to \widehat{(\cO_H\otimes\Z_p)^\times}\otimes_{\Z_p} \Q_p$,  is equivalent to its $\varrho$-parts  for  
$\varrho$ running over all  irreducible representations of $G$ contributing to $U_H$, justifying the  terminology.

\begin{lemma} \label{rho-Leo}
The following are equivalent: 
\begin{enumerate}[(i)]
\item the $\varrho$-part of the Leopoldt Conjecture holds; 
\item $\loc_p : \rH^1_f(V) \to \rH^1_f(\Q_p ,V)$ is injective; 
\item $\dim W^\circ=d^-$, where	 $ W^\circ=\left\{ w\in W\, \mid \, \log_p (\kappa(w))=0 \text{ for all } \kappa\in \Hom_G(W, U_H)\right\}$. 	 
\end{enumerate} 
\end{lemma}	 
Note that,  in general, the $\overline{\Q}_p$-linear subspace  $ W^\circ$ of $W$ is \emph{not} $G_p$-stable. 

By  Theorem~\ref{thm:BB} the $\varrho$-part of the Leopoldt Conjecture holds when $d^+\leqslant 1$. 
	 
\subsection{Regular submodules and $p$-refinements}
Let  $\chi : \rG_{\Q_p } \twoheadrightarrow \Gal(\Q_p (\mu_{p^\infty})/\Q_p ) \simeq \Z^\times_p \subseteq \bar\Q^\times_p$  be the $p$-adic cyclotomic character and  let us fix an arithmetic Frobenius element $\Frob_p$.

As $\varrho$ is unramified at $p$,  its restriction to $\rG_{\Q_p }$ is a direct sum of unramified finite order  characters $\chi_1,\ldots,\chi_d$, so that  one has $V_{\vert \rG_{\Q_p }}\simeq \chi\cdot\chi_1^{-1}\oplus \ldots \oplus \chi\cdot\chi_d^{-1}$. 
 
\subsubsection{The Bloch-Kato logarithm}\label{sec:the_BK_logarithm} 
It follows from the  definition of $B_\crys$ (see \cite{fontaine1982}) that it contains an element $t$ such that 
$\left(\Q_p (1) \otimes B_\crys \right)^{\rG_E}=t^{-1}E$ as $G_p$-modules, hence 
\[D_\crys(V)=\left(V \otimes_{\Q_p } B_\crys \right)^{\rG_{\Q_p }}=\Hom_{G_p}\left(W,t^{-1} E \otimes_{\Q_p } \overline{\Q}_p\right)
\overset{\mathrm{\cdot t }}{=\joinrel=}
\Hom(W, \overline{\Q}_p)
.\]
Moreover, the crystalline Frobenius $\varphi$ acts semi-simply on $D_\crys(V)$ with eigenvalues $(p^{-1}\alpha_i)_{1\leqslant i\leqslant d}$, where $\alpha_i=\chi_i(\Frob_p) \in \overline{\Q}^\times_p$.

Using Kummer theory, the  Bloch--Kato logarithm (see \cite[\textsection3]{bloch-kato}) for the $\rG_E$-module $\overline{\Q}_p(1)$ is simply given by: 
\[\log_p : \cO_E^{1} \otimes_{\Z_p } \overline{\Q}_p\overset{\sim}{\longrightarrow} E \otimes_{\Q_p } \overline{\Q}_p\quad x\otimes y\mapsto \log_p(x)\otimes y.\]
As  $V$ is crystalline and $\Fil^0D_\dR(V)=\{0\}$, the $p$-adic tangent space of $V$ is given by  
\[t_V=D_\dR(V)/\Fil^0D_\dR(V)=D_\crys(V).\]
 Using the identifications of \eqref{eq:frob-rec} and 
Lemma~\ref{lem:BK_desc}(i), the Bloch--Kato logarithm $\log_{\BK}$ for $V_{\vert \rG_{\Q_p }}$ fits in  the following commutative diagram
\begin{equation}\label{eq:log-BK}
\xymatrix{
\rH^1_f(\Q_p , V) \ar[rr]^{\log_\BK}_\sim \ar@{=}[d] &  &D_\crys(V) \ar@{=}[d] \\
\Hom_{G_p}(W,\cO_E^{1}\otimes_{\Z_p } \overline{\Q}_p) \ar[rr]^{\quad\log_p}_{\quad\sim} & &\Hom(W, \overline{\Q}_p).}
\end{equation}

Using the identification of Lemma~\ref{lem:BK_desc}(ii),
the $p$-adic regulator is the map $\log_\BK \circ \loc_p$ fitting in   the following  commutative diagram
\begin{equation}\label{eq:loc-log}
\xymatrix{
	\rH^1_f(V) \ar[rr]^{\log_\BK \circ \loc_p} \ar@{=}[d] & & D_\crys(V) \ar@{=}[d] \\
	\Hom_{G}\left(W,U_H\right) \ar[rr]^{\quad\log_p\circ \iota_p} & &\Hom(W, \overline{\Q}_p).}
\end{equation}

Note that $\rH^1_f(V)$ and $\rH^1_{f,p}(V)$ are endowed with canonical $\overline{\Q}$-structures given respectively by 
$\Hom_{G}\left(W_{\overline{\Q}},U_{H, \overline{\Q}}\right)$ and $\Hom_{G}\left(W_{\overline{\Q}},U^{(p)}_{H, \overline{\Q}}\right)$.

\subsubsection{Regular submodules} According to Perrin-Riou, the conjectural $p$-adic $L$-function of a motive should depend on the choice of local data at $p$, namely a regular submodule (see \cite[\textsection3.1.2]{perrin-riou:Lp} and \cite[\textsection4.1.3]{benoiscrys}).

\begin{definition}\label{def:module_reg}
	A submodule $D$ of $D_\crys(V)$, stable by $\varphi$, is regular if the map \[r_D : \rH^1_f( V ) \longrightarrow D_\crys(V)/D\]
	induced by the $p$-adic regulator map $\log_\BK \circ \loc_p$ is an isomorphism. 
\end{definition}
\begin{remark}
	
The existence of a regular submodule of $D_\crys(V)$ implies, in particular, that the localization map $\loc_p$ of Lemma~\ref{lem:BK_desc}(iii) is injective.
	
\end{remark}

\subsubsection{$p$-refinements} \label{sec:refinements}
Using \eqref{eq:log-BK}, any submodule $D$ of $D_\crys(V)$, stable by $\varphi$, can be written as 
\begin{equation}\label{eq:module_reg}
D=\Hom(W/W^+,\overline{\Q}_p), 
\end{equation}
for some $G_p$-stable linear subspace $W^+$ of $W$. Note that the action of $\varphi$ on $D$ corresponds, under (\ref{eq:module_reg}), to the action of $p^{-1}\cdot \Frob^{-1}_p$ on $\Hom(W/W^+,\overline{\Q}_p)$. Moreover, if $D$ is regular (see Def.~\ref{def:module_reg}) then $W^+$ has to be of dimension $d^+$,  motivating the following definition. 

\begin{definition}\label{def:p-refinement}
\begin{enumerate}[(i)]
\item 	A  \emph{$p$-refinement} of $W$ is a $d^+$-dimensional $G_p$-stable linear subspace $W^+$ of $W$. 
\item A $p$-refinement $W^+$ is \emph{regular} if  the submodule $D$ given by \eqref{eq:module_reg} is regular in the sense of Definition~\ref{def:module_reg}.
\item A $p$-refinement $W^+$ is  \emph{motivic} if there exists a $\overline{\Q}$-linear subspace $W^+_{\overline{\Q}}$ of $W_{\overline{\Q}}$ such that $W^+= W_{\overline{\Q}}^+\otimes_{\overline{\Q}} \overline{\Q}_p$. 
\item 	
	Given  any $p$-refinement $W^+$ one can attach a $p$-adic regulator of size $d^+$ as follows. Choose a basis $(\kappa_1,\ldots,\kappa_{d^+})$ of $\Hom_G(W,U_H)$ and a basis $(w^+_1,\ldots,w^+_{d^+})$ of $W^+$, and put 
	\[\Reg_p(W,W^+) = \det\left(\log_p \left(\kappa_j(w^+_i)\right)\right)_{1\leqslant i,j \leqslant d^+} \in \overline{\Q}_p. \]
	When $W^+$ is motivic one can use $\overline{\Q}$-bases of $\rH^1_f(V)_{\overline{\Q}}= \Hom_{G}\left(W_{\overline{\Q}},U_{H, \overline{\Q}}\right)$ 
	and $W^+_{\overline{\Q}}$.  	 
\end{enumerate}
\end{definition}

Recall the subspace 
$W^\circ= \bigcap_{i=1}^{d^+} \ker \left(\log_p \circ\kappa_i\right)$ from Lemma~\ref{rho-Leo}.

\begin{proposition} \label{prop:D_est_regulier}
\begin{enumerate}[(i)]
\item If $W$ admits a regular $p$-refinement then the $\varrho$-part of the Leopoldt Conjecture holds
and $\dim W^\circ =d^-$.  
\item  A $p$-refinement $W^+$ of $W$ is regular if any (hence all) of the following holds:
\begin{itemize}
\item $\Reg_p(W,W^+)\ne 0$, 
\item $\theta_D\ \colon \ \Hom_{G}\left(W,U_H\right)\to \Hom_{G_p}\left(W^+,\cO_E^{1}\otimes_{\Z_p } \overline{\Q}_p\right)$
is an isomorphism, 
\item $W=W^+\oplus W^\circ$.  
\end{itemize}
\item 
A motivic $p$-refinement is always regular if $d^+=1$,  or if we assume the Weak $p$-adic Schanuel Conjecture \cite[Conj.~3.10]{mazur-calegari}.
\end{enumerate}
\end{proposition}
\begin{proof} (i)  follows from  Lemma~\ref{rho-Leo}, as regularity implies the injectivity of $\loc_p$. 

(ii) By definition  $W^+$ is regular  if the map $r_D$,  for $D$ as in  \eqref{eq:module_reg}, 
is an isomorphism.
Using \eqref{eq:loc-log} and the canonically isomorphism between $D_\crys(V)/D$ and $\Hom(W^+,\overline{\Q}_p)$, we  are  lead to study the  map 
\begin{equation}\label{eq:reg-ref}
r_D: \Hom_G(W,U_H)  \longrightarrow \Hom(W^+,\overline{\Q}_p),
\end{equation}
 whose determinant, computed with respect to our fixed bases, is 
 precisely $\Reg_p(W,W^+)$. Thus, $r_D$ is an isomorphism if and only if $\Reg_p(W,W^+)\ne 0$. 
 
 As the map (\ref{eq:reg-ref}) is obtained by composing $\theta_D$ with the Frobenius reciprocity  isomorphism $\Hom_{G_p}(W^+,\cO^1_E \otimes_{\Z_p } \overline{\Q}_p) \simeq \Hom(W^+,\overline{\Q}_p)$ induced by $\log_p$, it follows that  $\theta_D$ is an isomorphism if and only if $\Reg_p(W,W^+)\ne 0$.

As the map (\ref{eq:reg-ref}) factors through $ \Hom(W/W^\circ,\overline{\Q}_p)$, if $W^+$ is regular then necessarily 
$W^+\cap W^\circ=\{0\}$ and therefore $W^+\oplus W^\circ=W$ (as $\dim W^\circ =d^-$ by (i)). Conversely, if 
$W^+\oplus W^\circ=W$, then Lemma~\ref{rho-Leo} implies that $\loc_p$ is injective, hence $r_D$ is an isomorphism, 
i.e.  $W^+$ is regular. 

(iii) Suppose that $W^+$ is motivic. Then one can assume the basis $w_1^+, \ldots,w_{d^+}^+$ to be chosen in $W_{\overline{\Q}}$. Moreover the $\kappa_i$'s can be chosen in $\rH^1_f(V)_{\overline{\Q}}$, so the entries of the regulator matrix are $p$-adic logarithms of elements of $U_{H,\overline{\Q}}$ which are easily seen to be linearly independent over $\Q$. Thus $\Reg_p(W,W^+)\ne 0$ by Theorem~\ref{thm:BB} if $d^+=1$, or else if  the Weak $p$-adic Schanuel Conjecture holds, as the entries of the determinant are linearly  independent over $\Q$. 
\end{proof}

\begin{remark}
		If the characters $(\chi_i)_{1\leqslant i\leqslant d}$ are pairwise distinct, then  a choice of a $p$-refinement of $W$ amounts to a choice of $d^+$ amongst these characters, and such a $p$-refinement is automatically motivic.  However, when the characters are \emph{not} pairwise distinct, then there are  infinitely many $p$-refinements some of which  are not motivic.
\end{remark}

\subsubsection{Extra zeros}\label{sec:extra zeros}
Let  $W^+$  be a $p$-refinement of $W$ and put $W^-=W/W^+$. Let $W_1^-=\rH^0(\Q_p ,W^-)$ be the subspace of $W^-$ on which the Frobenius acts trivially and consider the (unique) $G_p$-stable decomposition 
\[W^-=W/W^+=W_1^-\oplus W_{-1}^-.\]
 We assume that $\Reg_p(W,W^+)\ne 0$, so that $D=\Hom(W^-,\overline{\Q}_p)$ is a regular submodule of $D_\crys(V)$ by Proposition~\ref{prop:D_est_regulier}. The resulting decomposition of the $\varphi$-module $D$ as a direct sum of 
\begin{equation} \label{eq:D1}
D_1=\Hom(W_1^-,\overline{\Q}_p)  \text{ and } D_{-1}=\Hom(W_{-1}^-,\overline{\Q}_p)  
\end{equation}
   coincides with the one introduced in \cite[\textsection2.1.4]{benois:L-invariant} and \cite[\textsection4.1.3]{benoiscrys} in the sense that $D_1=D^{\varphi=p^{-1}}$ and $D_{-1}=(1-p^{-1}\varphi^{-1})D$.  As explained in  \emph{loc. cit.}, the $p$-adic $L$-function attached to (lattices of) $V$ and $D$ should have $e$ extra zeros at $s=0$, where $e=\dim D_1=\dim W_1^-$.

\subsection{The $\sL$-invariant} \label{sec:construction_L_invariant}
We now unwind Benois' definition of the $\sL$-invariant in the setting of Artin motives. Henceforth, we fix a 
regular submodule $D$ of $D_\crys(V)$, stable by $\varphi$, and we let 
$W^+$ denote the corresponding regular $p$-refinement  of $W$ (see \eqref{eq:module_reg}). 

\subsubsection{$(\varphi,\Gamma)$-modules} 
Benois' definition of the $\sL$-invariant involves the $(\varphi,\Gamma)$-modules 
$\textbf{D}$, $\textbf{D}_1$ and  $\textbf{D}_{-1}$, associated by Berger's theory \cite{berger2008} with $D$, $D_1$ and $D_{-1}$ (see \eqref{eq:D1}).  In our setting, these $(\varphi,\Gamma)$-modules come from local Galois representations, so their cohomology (in the sense of Herr \cite{herr1998}) is easily computable by \cite[Corollary 3]{benoiscrys}. From Lemma~\ref{lem:BK_desc}(i) one obtains:
\begin{align*}
\rH^1(\textbf{D}) &=\rH^1(\Q_p ,V^-) = \Hom_{G_p}\left(W^-, \widehat{E^{\times}} \otimes_{\Z_p } \overline{\Q}_p \right),\\
\rH_f^1(\textbf{D}) &=\rH_f^1(\Q_p ,V^-) = \Hom_{G_p}\left(W^-, \cO_E^{1} \otimes_{\Z_p } \overline{\Q}_p \right), \\
\rH_f^1(\textbf{D}_{-1})&=\rH^1(\textbf{D}_{-1}) =\rH^1(\Q_p ,V_{-1}^-) = \Hom_{G_p}\left(W_{-1}^-, \widehat{E^{\times}} \otimes_{\Z_p } \overline{\Q}_p \right),\\
\rH^1(\textbf{D}_1) &=\rH^1(\Q_p ,V_1^-) = \Hom_{G_p}\left(W_1^-, \widehat{E^{\times}} \otimes_{\Z_p } \overline{\Q}_p \right),\\
\rH_f^1(\textbf{D}_1)& =\rH_f^1(\Q_p ,V_1^-) = \Hom_{G_p}\left(W_1^-, \cO_E^{1} \otimes_{\Z_p } \overline{\Q}_p \right), 
\end{align*}
where  $V^{\pm}=(W^{\pm})^*(1)$ and where $V_1^-$ and $V_{-1}^-$ are similarly defined. 

As $\widehat{E^{\times}} = \cO_E^{1} \times p^{\Z_p }$, the map $ U_H^{(p)} \to \widehat{E^{\times}}\otimes \overline{\Q}_p$ induced by $\iota_p$ gives rise to a map
\footnotesize
	\begin{equation*}
	j_D : \Hom_G(W,U_H^{(p)}) \to\dfrac{\Hom_{G_p}(W,\widehat{E^{\times}}\otimes \overline{\Q}_p)}{\Hom_{G_p}(W^-,\cO_E^1 \otimes \overline{\Q}_p)} = \Hom_{G_p}\left(W^+,\cO_E^{1}\otimes_{\Z_p } \overline{\Q}_p\right)\oplus \Hom_{G_p}\left(W,p^{\overline{\Q}_p}\right),
\end{equation*}\normalsize
which can be seen to coincide with the map $\kappa_D$ in \cite[\textsection4.1.3]{benoiscrys}.

\begin{lemma}\label{lem:kappaD}
The map $j_D$ is an isomorphism. 
\end{lemma}
\begin{proof}
 By Frobenius reciprocity one finds
\[\Hom_G(W,U_H^{(p)})=\Hom_G(W,U_H)\oplus \Hom_{G_p}(W,p^{\overline{\Q}_p}),\] and the restriction of $j_D$
on the first summand is given by the map $\theta_D$ from Proposition~\ref{prop:D_est_regulier}(ii). 
Thus,  the regularity of $D$ implies that $j_D$ is an isomorphism. 
\end{proof}

\subsubsection{The space $\rH^1(D,V)$}\label{sec:def_H1(D,V)} 
Let  $\rH^1(D,V)$ be  the  subspace of $\rH^1_{f,p}(V)=\Hom_G(W,U_H^{(p)})$ (see Lemma~\ref{lem:BK_desc}(ii)) 
consisting of $\kappa : W \to U_H^{(p)}$ such that $\iota_p\circ\kappa$ is trivial on $W^+$, i.e.
$\iota_p\circ\kappa\in 
\Hom_{G_p}(W^-,\widehat{E^{\times}}\otimes \overline{\Q}_p)$. 

We describe the isomorphisms $j_{D,f}$ and $j_{D,c}$ 
(denoted by $\varrho_{D,f}$ and $\varrho_{D,c}$ in \cite[\textsection4.1.4]{benoiscrys}) following  \cite[\textsection1.5.6]{benois:L-invariant}. As $j_D$ is an isomorphism, its restriction to $\rH^1(D,V)$  yields
\[ j_{D,c}: \rH^1(D,V) \underset{\sim}{\xrightarrow{j_D}} \Hom_{G_p}\left(W^-,p^{\overline{\Q}_p}\right)
\underset{\sim}{\xrightarrow{\ord_p}} \Hom\left(W_1^-,\overline{\Q}_p\right),\]
Thus  $\rH^1(D,V)$ has dimension $e$. Denoting by $\mathrm{res}_{W_1^-}$ the restriction to $W^-_1$, the map $j_{D,f}$ fits in the following commutative diagram:
\[ \xymatrix{ \rH^1(D,V)  \ar[rr]^{\hspace{-15mm}\mathrm{res}_{W_1^-}\circ \iota_p} \ar[rrrd]_ {j_{D,f}} & &
\Hom_{G_p}\left(W_1^-, \widehat{E^{\times}} \otimes_{\Z_p } \overline{\Q}_p \right) \ar[r] &
\Hom_{G_p}\left(W_1^-,  \cO_E^{1} \otimes_{\Z_p } \overline{\Q}_p \right)  \ar@{=}[d]^{-\log_p} \\
& & &\Hom\left(W_1^-,\overline{\Q}_p\right).}
\]
 
 The following definition is taken from \cite[\textsection4.1.4]{benoiscrys}. 
\begin{definition}
	Benois' $\sL$-invariant is defined as the following $e \times e$ determinant: \[\sL(V,D)=\det\left(j_{D,f} \circ j_{D,c}^{-1} \ \vert  \ D_1\right) \in \overline{\Q}_p.\]
\end{definition}

\subsection{Computation of the $\sL$-invariant}\label{sec:computation_L_invariant}
Benois' $\sL$-invariant $\sL(V,D)$ is described in the last paragraph in terms of an $e$-dimensional space $\rH^1(D,V)$ whose definition is not always practical. In this paragraph, we give an alternate formula for $\sL(V,D)$ involving a determinant of size $(d^++e)$ which turns out to be especially useful in the examples of \textsection\ref{sec:gross}, \textsection\ref{reg-general} and \textsection\ref{sec:ad_reg}.

\begin{lemma}\label{lem:def_psi'}
Let $W_1^{+}=\rH^0(\Q_p ,W^+)$. The   $\overline{\Q}_p$-linear map
	\[\Hom_{G}\left(W,U_H^{(p)}\right)  \longrightarrow \Hom_{G_p}(W^+,\widehat{E^\times} \otimes_{\Z_p } \overline{\Q}_p) \longrightarrow \Hom_{G_p}\left(W^{+}, \overline{\Q}_p \right)=\Hom \left(W_1^{+}, \overline{\Q}_p \right)\]
obtained by localizing, restricting to $W^+$ and then composing with $\ord_p$ is surjective. Its kernel contains $\Hom_{G}\left(W,U_H\right)$ and has dimension $d^++e$ over $\overline{\Q}_p$. 
\end{lemma}
\begin{proof}
The first map is surjective because of the surjectivity of $j_D$, while  the second map is surjective  because it is induced by $\ord_p : E^\times \twoheadrightarrow \Z$. Thus the composed map is surjective and its kernel has dimension $(d^++f) -(f-e)=d^++e$ (see \textsection\ref{sec:dim}).
\end{proof}
\begin{notation}\label{nota:regulateur}
Let us fix a basis $(w^+_1,\ldots,w^+_{f-e})$ of $W_1^+$
which we complete to a basis $(w^+_1,\ldots,w^+_{d^+})$ of $W^{+}$, and also to a basis $(w^+_1,\ldots,w^+_{f-e},w^-_1,\ldots,w^-_{e})$ of 
$W_1=\rH^0(\Q_p,W)$. Note that $(w^-_1,\ldots,w^-_{e})$ maps onto a basis of $W^-_1$ under the natural projection map $W \twoheadrightarrow W^-$.

 Fix also a basis $(\kappa_1,\ldots, \kappa_{d^+})$ of $\Hom_{G}\left(W,U_H\right)$ which we complete to a basis $(\kappa_1,\ldots, \kappa_{d^+},\kappa'_1,\ldots, \kappa'_{e})$ of  the kernel of the map considered in Lemma~\ref{lem:def_psi'}.
	Our choices allow us to construct global units and $p$-units in the following way. 
	\begin{itemize}
		\item 
	 The global units  $\varepsilon_{i,j}^+=\kappa_j(w^+_i)$ and $\varepsilon_{i,j}^-=\kappa_j(w^-_i)$ belong to   $U_H$, and taking $p$-adic logarithms yields two matrices \[A^+=\left[\log_p(\varepsilon_{i,j}^+)\right]_{1\leqslant i,j \leqslant d^+} \textrm{  and  }A^-=\left[\log_p(\varepsilon_{i,j}^-)\right]_{\substack{1\leqslant i \leqslant e\phantom{^+} \\ 1\leqslant j \leqslant d^+}}\] whose entries are in $\overline{\Q}_p$. 
	\item 	 The global $p$-units $u_{i,j}^+=\kappa'_j(w^+_i)$ and $u_{i,j}^-=\kappa'_j(w^-_i)$ belong to  $U_H^{(p)}$, and taking $p$-adic logarithms yields two matrices \[B^+=\left[\log_p(u_{i,j}^+)\right]_{\substack{1\leqslant i \leqslant d^+ \\ 1\leqslant j \leqslant e\phantom{^+}}}\textrm{   and   }B^-=\left[\log_p(u_{i,j}^-)\right]_{1\leqslant i,j \leqslant e}\] whose entries are in $\overline{\Q}_p$.
		We  put $O^-=\left[\ord_p(u_{i,j}^-)\right]_{1\leqslant i,j \leqslant e}$. 
    \end{itemize}
\end{notation}

Note that $\det A^+$ is precisely the $p$-adic regulator $\Reg_p(W,W^+)$ from 
Definition~\ref{def:p-refinement}(iv), and the regularity of $W^+$ implies that 
 $\det A^+\ne 0$ by Proposition~\ref{prop:D_est_regulier}(ii). 

\begin{theorem}\label{thm:calcul_L_inv}
Assume $\Reg_p(W,W^+)$ is non-zero, so that the $\varphi$-submodule $D$ given by \eqref{eq:module_reg} is regular. The $\sL$-invariant associated with $V$ and $D$ is given by the following formula:\[\sL(V,D)=(-1)^e \frac{\det\begin{pmatrix}
	A^+ & B^+ \\ A^- & B^-
	\end{pmatrix} }{\det A^+ \cdot \det O^-}\]
where $A^\pm$, $B^\pm$ and $O^-$ are the matrices defined in Notation~\ref{nota:regulateur}. 
\end{theorem}

\begin{proof}
We first define convenient bases $\widetilde{\kappa}_1,\ldots,\widetilde{\kappa}_e$ and $\eta_1,\ldots,\eta_e$ of $\rH^1(D,V)$ and $D_1$ respectively and then express the matrices of $j_{D,f}$ and $j_{D,c}$ in these bases in terms of $A^\pm,B^\pm$ and $O^-$.

Let us complete  $(\kappa_1,\ldots, \kappa_{d^+},\kappa'_1,\ldots, \kappa'_{e})$ into a $\overline{\Q}_p$-basis $(\kappa_1,\ldots, \kappa_{d^+},\kappa'_1,\ldots, \kappa'_{f})$ of $\Hom_G(W,U_H^{(p)})$ such that the image of $(\kappa'_{e+1},\ldots, \kappa'_{f})$ under the surjective map in Lemma~\ref{lem:def_psi'} is the dual basis of $(w^+_1,\ldots,w^+_{f-e})$ of $\left(W_1^+\right)^*$. In other words,  for $1\leqslant i \leqslant d^+$ and  $e+1\leqslant j \leqslant f$, we have 
\[\ord_p(\kappa'_j(w^+_i))=\begin{cases} 1, \text{ if } e+i=j, \\  0, \text{ otherwise}.\end{cases}\]

Take $\kappa \in \rH^1_{f,p}(V)$ and write $\kappa= \sum_{i=1}^{d^+}c_i\kappa_i + \sum_{j=1}^{f}c'_j\kappa'_j$ for $c_j,c'_j\in \overline{\Q}_p$. We have $\kappa \in \rH^1(D,V)$ if and only if $\log_p\left(\kappa(w^+_i)\right)= \ord_p\left(\kappa(w^+_i)\right)=0$, i.e.
\[c'_{e+1}=\ldots=c'_f=0, \text{ and } A^+\cdot C_\kappa + B^+ \cdot C_\kappa' = 0, \]
where $C_\kappa$ and $C'_\kappa$ are the column vectors whose entries are $(c_1,\ldots,c_{d^+})$ and $(c'_1,\ldots,c'_e)$ respectively. As $A^+$ is invertible,  for all  $\kappa\in \rH^1(D,V)$ we have 
\[C_\kappa=-(A^+)^{-1}B^+C_\kappa'.\]
It follows that  $\rH^1(D,V)$ admits a basis $(\widetilde{\kappa}_1,\ldots,\widetilde{\kappa}_e)$ defined by 
\[\widetilde{\kappa}_j= \left(\sum_{i=1}^{d^+}c_{i,j}\kappa_i\right) - \kappa'_j, \]
where $(c_{1,j},\ldots,c_{d^+,j})$ is the $j$-th column of the matrix $(A^+)^{-1}B^+$, $1\leqslant j \leqslant e$. Letting  $(\eta_1,\ldots,\eta_e)$ denote the basis of $D_1=\Hom_{G_p}\left(W_1^-,\overline{\Q}_p \right)$ which is dual to $(w^-_1,\ldots,w^-_e)$, for  $1\leqslant j \leqslant e$, one has
 \[\begin{array}{l} j_{D,f}(\widetilde{\kappa}_j) = \sum_{i=1}^e \left(\log_p(u_{i,j}^-) -c_{1,j} \log_p(\varepsilon_{i,1}^-) -\ldots-c_{d^+,j}\log_p(\varepsilon_{i,d^+}^-) \right)\eta_i, \\ 
j_{D,c}(\widetilde{\kappa}_j) = - \sum_{i=1}^e \ord_p(u_{i,j}^-)\eta_i.\end{array}\] 
The matrices of $j_{D,f}$ and $j_{D,c}$ in our distinguished bases are thus given by $ B^- - A^-(A^+)^{-1}B^+ $ and $-O^-$ respectively, and the claimed formula follows by 
a computation of a determinant of a block matrix using the fact that $A^+$ is invertible. 
\end{proof}

\begin{remark}\label{rem:simplified_expression_L_inv}
\begin{enumerate}[(i)]
	\item One easily checks that the formula given in Theorem~\ref{thm:calcul_L_inv} does not depend on the choice of bases of $W^+$ and $W_1$. Furthermore, one always can find suitable $\overline{\Q}_p$-bases in which $\det O^-=1$.
	\item 
	Assume that  $\Frob_p$ acts  trivially on $W$, i.e. $W_1=W$ and $e=d^-$, and choose  $(w^-_1,\ldots,w^-_e)$ to be a basis of the linear subspace $W^\circ$ from  Lemma~\ref{rho-Leo}, which by Proposition~\ref{prop:D_est_regulier} is a supplement of 	$W^+$.
	Then   $A^-$ is the zero matrix and the above formula becomes $\sL(V,D)=(-1)^e\frac{\det B^-}{\det O^-}$.  
	\item 
	In view of the results in this subsection, the effective computation of $\sL(V,D)$ 
	relies a priori on that of $(\kappa_1,\ldots,\kappa_{d^+},\kappa'_1,\ldots,\kappa'_e)$, which  amounts to computing the $\varrho$-isotypic components of $U_H$ and of $U_H^{(p)}$. 	
	This could be achieved by applying an idempotent (with algebraic coefficients) to a Minkowski unit $\varepsilon_0$ in $U_H$, or more generally to  any  element $\varepsilon_0$ generating $U_H$ as a $G$-module. In the same vein, since $U_H^{(p)}/U_H\simeq \Ind_{G_p}^G \overline{\Q}_p$ by Dirichlet's Unit Theorem, any sufficiently generic element $u_0$ in $U_H^{(p)}$ generates $U_H^{(p)}/U_H$ as a $G$-module and can be used to determine the $\kappa'_j$'s. In contrast,  
	determining a basis of $\rH^1(D,V)$ would involve linear equations with 
	possibly transcendental coefficients. 	
	
	The situation is particularly nice when 
	$d^+=e=f=1$ and  $W^+=W^{\Frob_p=\beta}$. Let $e_\rho\in \overline{\Q}[G]$  be the idempotent attached to $\rho$ and define $e_{W^+},e_{W_1}\in \overline{\Q}[G_p]$ as \[e_{W^+}=\sum_{i=1}^{\vert G_p\vert }\beta^{-i}\cdot \Frob_p^{i} \quad \mbox{and} \quad e_{W_1}=\sum_{i=1}^{\vert G_p\vert}\Frob_p^{i}.\] 
	We can obtain $\varepsilon_{1,1}^\pm$ from $\varepsilon_0$ by letting $\varepsilon_{1,1}^+=e_{W^+}\cdot e_\rho\cdot\varepsilon_0$ and $\varepsilon_{1,1}^-=e_{W_1}\cdot e_\rho\cdot\varepsilon_0$ and similarly we can take $u_{1,1}^+=e_{W^+}\cdot e_\rho\cdot u_0$  and $u_{1,1}^-=e_{W_1}\cdot e_\rho\cdot u_0$.
		Such  numerical  computations have been    performed for certain $3$-dimensional representations in \cite{DLR2}. 
\end{enumerate}	 
\end{remark}

\subsection{The dual $\sL$-invariant}\label{sec:dualL}
 In this section, we give an alternative expression of the $\sL$-invariant that derives from a dual construction given in \cite[\textsection2.3]{benoishorte} which will be used in \textsection\ref{adjoint-CM-irreg}. Perrin-Riou's definition of regularity actually applies to $D_\crys(W)$ and one checks that a submodule $D\subseteq D_\crys(V)$ is regular if and only if its orthogonal complement $D^\perp$ for the natural pairing $D_\crys(V) \times D_\crys(W) \to D_\crys(\overline{\Q}_p(1))$ is a regular submodule of $D_\crys(W)$. We now fix a regular submodule $D$ of $D_\crys(V)$ and we keep the same notations as in the previous sections. 

By Class Field Theory and Frobenius reciprocity one has 
\[\rH^1_{f,p}(W) = \ker\left(\Hom_{G_p}\left(E^{\times},W\right)\longrightarrow \Hom_G({\cO_H\left[\tfrac{1}{p}\right]^\times},W)\right).\]
Consider the linear subspace $\rH^1(D^\perp,W)$ of $ \rH^1_{f,p}(W)$, 
consisting of elements $\kappa$ such that $\kappa(E^{\times})\subset W^+ + W_1$.  Computations dual to those in Section~\ref{sec:construction_L_invariant} show that $\rH^1(D^\perp,W)$ is the module introduced by Benois-Horte in \cite[\textsection2.3.6]{benoishorte} and that it has dimension $e$ over $\overline{\Q}_p$.
For any $\kappa \in \rH^1(D^\perp,W)$ we denote by $\kappa^-$ its composition with the 
projection $(W^+ + W_1)\twoheadrightarrow W_1^-$   and we put 
\[j_{D^\perp,f}(\kappa)=\kappa^-(p), \qquad j_{D^\perp,c}(\kappa)=\frac{\kappa^-(1+p)}{\log_p(1+p)}.\]
This defines two linear maps $j_{D^\perp,f},j_{D^\perp,c}: \rH^1(W,D^\perp) \to W_1^-$ such that  $j_{D^\perp,c}$ is an isomorphism. 
When specialized to our setting, the  construction of the dual $\sL$-invariant $\sL(W,D^\perp)$ given in \cite[Definition 2.3.7]{benoishorte},  leads to the following formula: 
\[\sL(W,D^\perp) = (-1)^e \det \left( j_{D^\perp,f} \circ \left(j_{D^\perp,c}\right)^{-1} \vert W_1^-\right).\]
The following proposition is \cite[Prop.~2.3.8]{benoishorte}. Note that the sign $(-1)^e$ corresponds to the sign in the conjectural functional equation for $p$-adic $L$-functions, as explained in detail in \cite[\textsection2.2.6, \textsection2.3.5]{benois:L-invariant}.
 \begin{proposition}\label{dual-L-inv}
 	Assume that $D$ is a regular submodule of $D_\crys(V)$. Then
 	\[\sL(W,D^\perp) = (-1)^e  \sL(V,D).\]
 	\end{proposition}

\section{Examples of computation of  $\sL$-invariant}\label{sec:examples}
\subsection{Gross' $p$-adic regulator}\label{sec:gross}
We assume throughout this paragraph that $d^+=0$, i.e. any complex conjugation in $\Gal(\overline{\Q}/\Q)$ acts on $W$ as $-1$. Then the Galois extension $H/\Q$ cut out by $\varrho$ is a CM field. Conversely, if $\varrho$ is irreducible and cuts out a CM field $H$, then 
any complex conjugation defines a non-trivial central element in $G=\Gal(H/\Q)$ which is necessarily scalar,  hence equal to $-1$, by Schur's lemma. The unique $p$-refinement $W^+=\{0\}$ of $W$ is automatically regular
 and the corresponding regular submodule is $D=D_\crys(V)$. Note that we have 
 \[e=\dim \rH^0(\Q_p ,W).\] 
One may attach to $\varrho$ a $p$-adic $L$-function by using Deligne-Ribet's construction \cite{deligne-ribet} together with Brauer's induction theorem (see \cite{greenbergartinII}). Gross \cite{gross81} has proposed a conjectural formula for the leading term of this $p$-adic $L$-function at $s=0$ in terms of a certain $\sL$-invariant which is nowadays referred to as Gross's $p$-adic regulator. We briefly recall its definition and relate it to Benois' $\sL$-invariant. 

Let $Y$ be the $\overline{\Q}_p$-vector space $\bigoplus_\gp \overline{\Q}_p\cdot\gp$, where the direct sum runs over the set $S_p(H)$ of all  primes $\gp$ of $H$ above $p$. We  denote by $X \subset Y$ the kernel of the augmentation map $\sum_\gp a_\gp \cdot \gp \mapsto \sum_{\gp} a_\gp$. 
 Note that $G$ naturally acts on both $X$ and $Y$ (see \textit{e.g.} \cite[Chap.~I~\textsection3]{tate}), and $\Hom_G(W,X)=\Hom_G(W,Y)$ as $Y=X\oplus \overline{\Q}_p$. 
For $\gp\in S_p(H)$, Gross \cite[(1.8)]{gross81} defines the local absolute value of $\alpha \in  H_\gp$ by letting $\vert \vert \alpha\vert \vert _{\gp,p}=p^{-\ord_\gp(\alpha)f_p}\textbf{N}_{H_\gp/\Q_p }(\alpha)$, where $f_p$ is  the residual  degree  of $H_\gp$ and $\textbf{N}_{H_\gp/\Q_p }$ is the norm map for the field extension $H_\gp/\Q_p$. Gross' $p$-adic regulator map is given by
\[\lambda_p: U_H^{(p)} \longrightarrow Y, \qquad 
	 u \mapsto \sum_{\gp\in S_p(H)} \log_p (\vert \vert u\vert \vert _{\gp,p})\cdot\gp, \]
whose image lies in $X$ by the product formula. 
Gross introduces another $G$-equivariant map 
\[o_p: U_H^{(p)} \longrightarrow Y, \qquad
u \mapsto \sum_{\gp\in S_p(H)} f_p\cdot \ord_{\gp}(u)\cdot\gp,\]
 (see \emph{loc. cit.}, (1.12) and (1.20)). Since $\dim \Hom_G(W,U_H)=d^+=0$, the map $\Hom_G(W,U_H^{(p)}) \to \Hom_G(W,Y)=\Hom_G(W,X)$ induced by $o_p$ is injective, and it is in fact an isomorphism of $e$-dimensional $\overline{\Q}_p$-vector spaces.
 Gross' regulator for $W^*=\Hom_{\overline{\Q}_p}(W,\overline{\Q}_p)$ is then defined as 
\[R_p(W^*)=\det\left(\lambda_p \circ o_p^{-1}\ \big\vert  \ \Hom_G(W,X)\right).\]
Here, we still denote by $\lambda_p$ and $o_p$ the maps obtained by post-composing by $\lambda_p$ and $o_p$ respectively. Note that, under the canonical identification $(W^*\otimes -)^G=\Hom_G(W,-)$, these maps coincide with the maps $1\otimes \lambda_p,1\otimes o_p \ \colon \ (W^*\otimes U_H^{(p)})^G \to (W^*\otimes X)^G$  appearing  in Gross' definition (\textit{loc. cit.}, (2.10)). 
\begin{proposition}\label{prop:gross_reg}
	Let $D=D_\crys(V)$. One has $\sL(V,D)= (-1)^e\cdot R_p(W^*)$.
\end{proposition}

\begin{proof}
	Let $\gp_0\in S_p(H)$ be the  prime above $p$ determined by $\iota_p$ and put $W_1=\rH^0(\Q_p ,W)$. Consider the following diagram 
	\[\xymatrix{
		\Hom_{G}(W,X) \ar[rr]^\pi && \Hom_{G_p}(W,\overline{\Q}_p) \ar@{=}[rr]&& \Hom(W_1,\overline{\Q}_p), \\
		&& \Hom_{G}(W,U_H^{(p)}) \ar[ull]_{\lambda_p}^{o_p} \ar[urr]^{\log_p}_{\ord_p}&&
	}\]
where the map $\pi$ is induced by the projection $\sum_{\gp} a_\gp \cdot \gp \in X \mapsto a_{\gp_0}$. By Frobenius reciprocity $\pi$ is an isomorphism because $W$ does not contain the trivial representation.  We claim that the above diagram becomes commutative when we multiply the maps $\log_p$ and $\ord_p$ by the same factor $f_p$, that is, the maps $\pi \circ \lambda_p$ and $\pi \circ o_p$ respectively coincide with $f_p\cdot \log_p$ and $f_p \cdot \ord_p$. 
Indeed, if $\kappa\in \Hom_G(W,U_H^{(p)})$, then $\kappa_{\vert W_1}$ takes values in $(U_H^{(p)})^{G_p}$, hence $\iota_p\circ \kappa_{\vert W_1}$ lands in $(H_{\gp_0}^\times
\otimes\overline{\Q}_p  
)^{G_p}=  \Q_p ^\times \otimes \overline{\Q}_p $. Our claim then follows from the fact that the map $\log_p \circ\vert \vert \cdot\vert \vert _{\gp_0,p}$ coincides with $f_p\cdot \log_p$ on $\Q_p^\times$. This proves that 
$R_p(W^*)$ is equal to the determinant of the map $\log_p \circ \ord_p^{-1}$ acting on $\Hom(W_1,\overline{\Q}_p)$. In terms of Notation~\ref{nota:regulateur}, it is then clear that $R_p(W^*)=\det(B^-)\cdot\det(O^-)^{-1}$, so $\sL(V,D)= (-1)^e\cdot R_p(W^*)$ by Theorem~\ref{thm:calcul_L_inv}.
\end{proof}

\subsection{Weight $1$ cuspforms}

Given a weight $1$ newform $f$, we let $\varrho_f$ denote the corresponding 
 Artin representation constructed by Deligne and Serre, which ramifies precisely at the primes dividing the level of $f$. 
  Since $\varrho_f$  is two-dimensional and 
odd, one has $d^+=d^-=1$, which can be thought of as a regularity condition at $\infty$. 
It might be worth recalling that the Khare--Wintenberger proof of Serre's Modularity Conjecture implies that all  
two-dimensional odd  Artin representations are obtained in this way. Since $\varrho_f$ has  finite image, one can  choose a $\overline{\Q}$-rational basis for its image and consider $\varrho_f$  as  $\overline{\Q}_p$-valued, via the fixed embedding $\iota_p: \overline{\Q}\hookrightarrow \overline{\Q}_p$. It is then crystalline at $p$ if and only if $p$ does not divide the level of $f$, which we will henceforth assume.

We say  that $f$ is $p$-regular if the eigenvalues  $\alpha$ and $\beta$ of $\varrho_f(\Frob_p)$  are distinct, and
$p$-irregular otherwise.

\subsubsection{The $p$-regular case} \label{reg-general}
Assume that $f$ is $p$-regular. The choice of a $p$-refinement $W^+$ (see Definition~\ref{def:p-refinement}) amounts to choosing between $\alpha$ and $\beta$. 
We let $f_\alpha$ denote the $p$-stabilization of $f$ having $U_p$-eigenvalue $\alpha$ and let  $W^+=W_\beta$ be
the line on which $\Frob_p$ acts by multiplication by $\beta\in \overline{\Q}^\times$. As $\alpha\ne\beta$, $W^+$ admits an $\overline{\Q}$-rational basis,  i.e. it is a motivic $p$-refinement in the sense of Definition~\ref{def:p-refinement}. It is thus a regular $p$-refinement by Proposition~\ref{prop:D_est_regulier}. The $\sL$-invariant $\sL(f_\alpha)$ of $f_\alpha$ is defined as  $\sL(V,D_\alpha)$ where $D_\alpha$ is the regular submodule attached to $W_\beta$ (see (\ref{eq:module_reg})). One checks that $e=0$ if $\alpha \ne 1$, in which case one has $\sL(f_\alpha)=1$, and that $e=1$ if $\alpha=1$. Thus, the conjectural Perrin-Riou $p$-adic $L$-function $L_p(f_\alpha,s)$ attached to $f_\alpha$, i.e. to choices of  lattices in $V$ and in $D_\alpha$, has a trivial zero at $s=0$ if and only if $\alpha=1$. A potential candidate for $L_p(f_\alpha,s)$ is constructed (up to a non-zero constant) by Bella\"iche and the first author in \cite{bellaiche-dimitrov} and  some of its properties are studied by the second author in \cite{maksoud}. 
It is namely shown   that it vanishes at $s=0$ when $\alpha=1$, as expected.
	
	Assuming  $e=1$, let $(w_\alpha,w_\beta)$ be a basis of $W$ consisting of eigenvectors for $\Frob_p$ and let $(\kappa,\kappa')$ be a $\overline{\Q}_p$-basis of $\Hom_G(W,U_H^{(p)})$ such that $\kappa$ takes values in $U_H$. These choices pin down two elements 
	$\varepsilon_\alpha=\kappa(w_\alpha)$ and $\varepsilon_\beta=\kappa(w_\beta)$ in 
	$U_H$ and two elements
		$u_\alpha=\kappa'(w_\alpha)$ and $u_\beta=\kappa'(w_\beta)$ in $U^{(p)}_H$, and  Theorem~\ref{thm:calcul_L_inv} yields:
		\[\sL(f_\alpha)= \frac{\log_p(\varepsilon_\alpha)\log_p(u_\beta)-\log_p(\varepsilon_\beta)\log_p(u_\alpha)}{\log_p(\varepsilon_\beta)\ \ord_p(u_\alpha)}.\]

\subsubsection{The $p$-irregular case}\label{irreg-general}

	 Assume that $f$ is $p$-irregular. Then $\Frob_p$ acts as the scalar $\alpha=\beta$ on $W$ and any line $W^+\ne  W^\circ$ of $W$ defines a regular $p$-refinement.  One has again $e=0$ (and hence, $\sL(V,D)=1$) if $\alpha\ne 1$ and $e=1$ if $\alpha=1$. Assuming $e=1$, one can compute  $\sL(V,D)$ as in Section~\ref{sec:computation_L_invariant}. We choose a basis $w^+$ of $W^+$ 
	 that we complete into a basis $(w^+,w^-)$ of $W$. Given  a  basis $\kappa$  of 
	 $\Hom_G(W,U_H)$, Lemma~\ref{lem:def_psi'} ensures the existence of 
	  $\kappa'\in \Hom_G(W,U_H^{(p)})$  non-proportional to $\kappa$ and such that $\ord_p(\kappa'(w^+))=0$. As in the $p$-regular case, these choices pin down two elements 
	$\varepsilon^\pm=\kappa(w^\pm)$  in 
	$U_H$ and two elements		$u^\pm=\kappa'(w^\pm)$  in $U^{(p)}_H$,   with the extra property that $\ord_p(u^+)=0$. By Theorem~\ref{thm:calcul_L_inv}, one has the following expression of the $\sL$-invariant:\[\sL(V,D)= \frac{\log_p(\varepsilon^-)\log_p(u^+)-\log_p(\varepsilon^+)\log_p(u^-)}{\log_p(\varepsilon^+)\ \ord_p(u^-)}.\]

\subsubsection{The CM case} \label{CM-case}
Let  $f$ be a weight one cuspform having complex multiplication by an imaginary quadratic field $K$. Then 
$\varrho_f=\Ind_\Q^K\psi$ with $\psi:\rG_K \to \overline{\Q}_p$   a  finite order character cutting out an abelian extension $H_\psi/K$.  As $\varrho_f$ is irreducible, $\psi$  is different from its  conjugate $\bar\psi=\psi(\tau\cdot \tau)$, allowing us to decompose 
 $W$ as a direct sum  $\overline{\Q}_p$-lines  $W_\psi$  and $W_{\bar\psi}$ on which $\rG_K$ acts by $\psi$ and ${\bar\psi}$, respectively.  Fix a basis $e_1$ of $W_\psi$ and put $e_2=\tau(e_1)\in W_{\bar\psi}$.
 The fixed field $H$ of $\ker(\varrho_f)$ is the compositum  of $H_\psi$ and $H_{\bar\psi}$, and we let
 $G=\Gal(H/\Q)$.   For any $\overline{\Q}_p[G]$-module $M$, Frobenius reciprocity  allows us to canonically identify $\Hom_G(W,M)$ with the $\psi$-isotypic component $M[\psi]$ of $M$. We now give a simpler expression for the $\sL$-invariants computed in \textsection\ref{reg-general}-\ref{irreg-general} when $p$ splits in $K$.

 We parameterize the lines of $W$ by $\mathbb{P}^1(\overline{\Q}_p)$ by letting  $W^+_s=\overline{\Q}_p(e_1+se_2)$ when $s\in  \overline{\Q}_p$ and
 $W_\infty^+=\overline{\Q}_p e_2$.
 Write $p\cO_K=\gp\bar\gp$, where $\gp$ is the prime determined by $\iota_p$ and $\bar\gp=\tau(\gp)$. Note that $\{\alpha,\beta\}=\{\psi(\gp),\psi(\overline{\gp})\}$.

 Given a basis $\varepsilon_\psi$ of the line $U_H[\psi]$ 
  the slope of $\psi$ is defined as   (see \cite[\textsection1.1]{betina-dimitrov})
 \begin{equation}\label{eq:def_slope}
\sS_\psi=-\frac{\log_p(\varepsilon_\psi)}{\log_p(\tau(\varepsilon_\psi))}\in \overline{\Q}_p^\times. 
 \end{equation}
 As $G_p \subseteq  \Gal(H/K)$, one can check that $W_s^+$ is a regular $p$-refinement if and only if 
 \begin{itemize} 
\item  either $\psi(\gp)\ne \psi(\bar\gp)$ and $s\in\{0,\infty\}$ (the $p$-regular case),  
\item or   $\psi(\gp)=\psi(\bar\gp)$ and $s \ne \sS_\psi$ (the $p$-irregular case). 
 \end{itemize}

  Let $D_s$ be the regular submodule associated with $W^+_s$ and assume that we are in presence of trivial zeros, i.e.  $\psi(\bar\gp)=1$ if $s=0$, $\psi(\gp)=1$ if $s=\infty$, or $\psi(\gp)=\psi(\bar\gp)=1$ in the $p$-irregular setting. Then $d^+=d^-=e=1$. 
  
  Assuming that $\psi(\bar\gp)=1$, the cyclotomic $\sL$-invariant of $\psi$ is defined 
as   
\begin{equation}\label{eq:def_L_inv_psi}
	\sL(\psi)=-\frac{\log_p(\tau(u^\circ_\psi))}{\ord_p(\tau(u^\circ_\psi))},
\end{equation}
where $u^\circ_\psi$ a basis of the line in $U_H^{(p)}[\psi]$ defined by the equations  $\log_p(\cdot)=\ord_p(\cdot)=0$
(see \cite[Rem.~1.5(iii)]{buyukboduk-sakamoto}). 
In order to avoid the presence of the possibly transcendental equation $\log_p(\cdot)=0$, we consider the plane in $U_H^{(p)}[\psi]$ cut out  simply by the  rational equation $\ord_p(\cdot)=0$
(the careful reader might notice the analogy with \textsection\ref{sec:computation_L_invariant} where $\rH^1(D_s,V)$ was  replaced by a more concrete space of a larger dimension). 
It is easy to see that given any $\bar\gp$-unit 
$u_\psi\in U_H^{(p)}[\psi]$ (which is not a unit), the set $\{\varepsilon_\psi, u_\psi\}$ forms a basis of this plane, and furthermore  $u^\circ_\psi$ from 
\eqref{eq:def_L_inv_psi} can be taken to be  $u_\psi\otimes \log_p(\varepsilon_\psi)-
\varepsilon_\psi\otimes \log_p(u_\psi)$. Thus one retrieves the more explicit formula from  \cite[(6)]{betina-dimitrov}
\begin{equation}\label{eq:def_L_inv_psi-bis}
	\sL(\psi)=-\frac{\log_p(\tau(u_\psi))+\sS_{\bar\psi}\log_p(u_\psi)}{\ord_p(\tau(u_\psi))}. 
\end{equation}

\begin{proposition} 
\begin{enumerate}[(i)]
\item In the $p$-irregular case, for all  $s\ne \sS_\psi$, we have the following formula
 \[\sL(V,D_s)= \frac{s\cdot \sL(\bar\psi)-\sS_\psi\cdot \sL(\psi)}{s-\sS_\psi}.\]
 	\item In the $p$-regular case one has  $\sL(V,D_0)=\sL(\psi)$,  if $\psi(\bar\gp)=1$,  and $\sL(V,D_\infty)=\sL(\bar\psi)$, if  $\psi(\gp)=1$. 
 \end{enumerate}
 \end{proposition} 
 
 \begin{proof} We apply the formalism of \textsection\ref{sec:computation_L_invariant} noting that
  a natural isomorphism between $\Hom_G(W,U_H^{(p)})$ and 
 $U_H^{(p)}[\psi]$ is obtained by evaluating at $e_1\in  W_\psi$. 
 Moreover the line $W^\circ$ from Lemma~\ref{rho-Leo} has basis $e_1+ \sS_\psi e_2$. 
 
 (i) Assume that $\psi(\gp)=\psi(\bar\gp)=1$ and fix $s\ne \sS_\psi$. 
A supplement of  the line  $\Hom_{G}\left(W,U_H\right)$ in the kernel of the map in 
 Lemma~\ref{lem:def_psi'} is  generated by an element $\kappa'_s\in \Hom_G(W,U_H^{(p)})$ such that $\ord_p(\kappa'_s(e_1+se_2))=0$.
 Rescaling  the units $u^\circ_\psi$  so that $\ord_p(\tau(u^\circ_\psi))=1$, one finds that the element 
 $\kappa'_s$ determined by  $\kappa'_s(e_1)=u^\circ_\psi \cdot \tau(u^\circ_{\bar\psi})^{-s}\in U_H^{(p)}$
 has all the required properties, and of course   
 $\kappa'_s(e_2)=\tau(u^\circ_\psi) \cdot (u^\circ_{\bar\psi})^{-s}$. Moreover, as $G_p$ acts trivially on $W^\circ$,   Remark~\ref{rem:simplified_expression_L_inv}(ii)
 applies and we have
 \[\sL(V,D_s)=-\frac{\log_p\left(\kappa'_s(e_1+ \sS_\psi e_2)\right)}{\ord_p\left(\kappa'_s(e_1+ \sS_\psi e_2)\right)}=-\frac{-s\cdot \log(\tau(u^\circ_{\bar\psi}))+ \sS_\psi \log(\tau(u^\circ_\psi))}{-s+\sS_\psi},\]
so the formula follows from \eqref{eq:def_L_inv_psi}.

(ii) By symmetry, it suffices to compute $\sL(V,D_0)$ when $\psi(\bar \gp)=1\ne \psi( \gp)$. Following Notation~\ref{nota:regulateur}, we have  $W^+=\overline{\Q}_p e_1$,  
$W_1=\overline{\Q}_p e_2$ and $W_1^+=\{0\}$, in particular the  map in Lemma \ref{lem:def_psi'} is the zero map.  Let $\kappa$ be a basis of $\Hom_G(W,U_H)$ that we complete in a basis 
 $(\kappa,\kappa'_0)$  of $\Hom_G(W,U_H^{(p)})$ such that  $\kappa_0'(e_1)=u_\psi^\circ$. Since $\log_p(\kappa'_0(e_1))=0$, i.e. $B^+=(0)$, Theorem \ref{thm:calcul_L_inv} yields
\[\sL(V,D_0)= - \frac{\log_p(\kappa'_0(e_2))}{\ord_p(\kappa'_0(e_2))}=\sL(\psi),\]
as claimed. 
\end{proof} 

\subsection{Adjoint of a weight $1$ $p$-regular cuspform}\label{sec:ad_reg}

Let $\varrho$ be the adjoint of the Artin representation $\varrho_f$ attached by Deligne and Serre to a  weight one cuspform $f$ of level relatively prime to $p$. As in the introduction, we assume that $\varrho_f$ has $\overline{\Q}$-coefficients and we view it via $\iota_p$ as a $p$-adic representation. 
Let $\alpha$ and $\beta$ be the eigenvalues of  $\varrho_f(\Frob_p)$ and assume  that $\alpha\ne \pm \beta$ (so $f$ is in particular $p$-regular). We let $(w_1,w_{\beta/\alpha},w_{\alpha/\beta})$ be a basis of $W$ on which $\Frob_p$ acts with the eigenvalue prescribed in the index and we denote by  $(w^*_1,w^*_{\beta/\alpha},w^*_{\alpha/\beta})$ its dual basis. The $p$-adic realization $V$ of the motive $M=(\ad^0(f))(1)$ has a unique decomposition as $\rG_{\Q_p }$-modules 
$V=\overline{\Q}_p(1)\cdot w^*_1 \bigoplus \overline{\Q}_p(1)\cdot w^*_{\beta/\alpha}\bigoplus \overline{\Q}_p(1)\cdot w^*_{\alpha/\beta}$.  

We now apply the results of \textsection\ref{sec:computation_L_invariant}.
Since one has $d^+=e=f=1$ and $\varrho(\Frob_p)$ acts with pairwise distinct eigenvalues, the line $W^+$ generated by $w_{\beta/\alpha}$ is a motivic $p$-refinement (so it is regular by Proposition~\ref{prop:D_est_regulier}(iii)) and $W_1$ is a line with basis $w_1$. Let $(\kappa,\kappa')$ be a $\overline{\Q}_p$-basis of $\Hom_G(W,U_H^{(p)})$ such that $\kappa$ takes values in $U_H$.
Without loss of generality, we may assume that $w_{\beta/\alpha},w_1$  (resp. $\kappa,\kappa'$) belong to the canonical $\overline{\Q}$-structure of $W$ (resp. of $\rH^1_{f,p}(V)$ described in \textsection\ref{sec:the_BK_logarithm}). By Theorem~\ref{thm:calcul_L_inv}, Benois' $\sL$-invariant for the regular submodule $D=\Hom(W/W^+,\overline{\Q}_p)$ is given by
\[\sL(V, D)= \frac{\log_p(\varepsilon_1)\log_p(u_{\beta/\alpha})-\log_p(\varepsilon_{\beta/\alpha})\log_p(u_1)}{\log_p(\varepsilon_{\beta/\alpha})\ord_p(u_1)},\]
where $\varepsilon_1=\kappa(w_1)$, $\varepsilon_{\beta/\alpha}=\kappa(w_{\beta/\alpha})$, $u_1=\kappa'(w_1)$ and $u_{\beta/\alpha}=\kappa'(w_{\beta/\alpha})$.

Let us assume that $\varrho_f$ is residually irreducible and $p$-distinguished, the latter condition meaning that $\alpha$ and $\beta$ are distinct modulo the maximal ideal of $\overline{\Z}_p$. We suppose that $\varrho_f$ is not induced by a character of a real quadratic field in which $p$ splits, hence by \cite{bellaiche-dimitrov} the eigencurve is \'{e}tale over the weight space at both $p$-stabilizations $f_\alpha$ and $f_\beta$ of $f$.
Using the automorphic deformation of $f_\alpha$, Hida defined an analytic $\sL$-invariant $\sL(\ad^0(f_\alpha))$ for which Darmon, Lauder and Rotger proposed in \cite[Conj. 1.2]{DLR2} a conjectural description in terms of the $p$-adic logarithms of $\varepsilon_1,\varepsilon_{\beta/\alpha},u_1,u_{\beta/\alpha}$ (see also \cite[\textsection5.1]{rivero-rotger}). In light of our computation of Benois' $\sL$-invariant in this setting, the main theorem of Rivero and Rotger \cite[Thm. A']{rivero-rotger} can then be recast as
\[\sL(V, D)= \sL(\ad^0(f_\alpha)) \mod \overline{\Q}^\times.\]

\subsection{Adjoint of a weight $1$   $p$-irregular CM cuspform} \label{adjoint-CM-irreg}

We consider the case of  the Artin motive $W=\ad^0\varrho_f$, where $f$ is a weight one cuspform 
 irregular at $p$ and having CM by an imaginary quadratic field $K$ in which $p$ splits. 
 Letting $\varrho_f=\Ind_K^\Q\psi$, where $\psi$ is  a finite order Hecke character of $K$, the above assumptions translate into the fact that the anti-cyclotomic (or ring class) character $\varphi=\psi/\bar\psi$ is non-trivial, but is locally trivial at (places above) $p$. Let $H$ denote the fixed field of $\ker(\varphi)$. 
 
 \subsubsection{Regular submodules}  Fix a dihedral basis of $\varrho_f$ as in \textsection\ref{CM-case}, so that  $\varrho_{f\vert \rG_K}=\psi\oplus \bar\psi$. This choice yields a matrix representation $\begin{pmatrix} a & b \\ c & -a \end{pmatrix}$ of $\ad^0 \varrho_f$ with basis $\{w_1,w_2,w_3\}$ given by 
 $w_1=\begin{pmatrix} 1 & 0 \\ 0 & -1\end{pmatrix}$ and 
  $w_2=\begin{pmatrix} 0 & 1 \\ 0 & 0\end{pmatrix}$, $w_3=\tau(w_2)=\begin{pmatrix} 0 & 0 \\ 1 & 0\end{pmatrix}$ 
  Moreover, this basis is adapted to the $\rG_\Q$-decomposition $\ad^0\varrho_f= \varepsilon_K  \oplus\Ind_K^\Q\varphi$, where $\varepsilon_K$ is the Dirichlet character associated with the quadratic extension $K/\Q$. If $\varphi$ is quadratic then  
 $\Ind_K^\Q\varphi=\varepsilon_{K'}\oplus\varepsilon_F$, where $K'$ and $F$ are the other two quadratic subfields of $H$. Otherwise  $\Ind_K^\Q\varphi$ is irreducible and we have 
$(\Ind_K^\Q\varphi)_{\vert \rG_K}=\varphi\oplus \bar\varphi$ in the dihedral basis $\{w_2,w_3\}$.

 By Lemma~\ref{lem:BK_desc}(ii) and Frobenius reciprocity one has  a natural isomorphism
\[\rH^1_f(V)=\Hom_G(W,U_H)=\Hom_G(\Ind_K^\Q\varphi,U_H)\xrightarrow{\sim}U_H[\varphi], \]
where $\kappa\in \Hom_G(W,U_H)$ sends $w_1$ to $1$, sends $w_2$ to  
a $\varphi$-unit $u_\varphi\in U_H[\varphi]$  and sends $w_3=\tau(w_2)$ to 
$\tau(u_\varphi)\in U_H[\bar\varphi]$.

We have $d=3$, $d^+=1$ and any line $W^+\subset W$ is $G_p$-stable.  
By Proposition~\ref{prop:D_est_regulier}(ii), $W^+$ is regular  if and only if it is not contained in the plane defined by 
the equation 
\[\log_p(u_\varphi)b+\log_p(\tau(u_\varphi))c=0 \text{, i.e.  }  \sS_\varphi\cdot  b=c, \]
where the slope 
$\sS_\varphi$ is defined in \eqref{eq:def_slope}.  In particular, the only non-regular submodule of
$\Ind_K^\Q\varphi$ is the line generated by $w_2+\sS_\varphi w_3$. 
 \subsubsection{A modular basis of $\rH^1(\Q,\ad^0\varrho_f)$}  \label{sec:basis}
Letting $\sL_\gp=-\frac{\log_p(u_\gp)}{\ord_p(u_\gp)}$ for any non-torsion element $u_\gp$ of $\cO_K[\tfrac{1}{\gp}]^\times$, the  anti-cyclotomic $\sL$-invariant  is defined as $\sLm(\varphi)=\sL(\varphi)-2 \sL_\gp$ (see \eqref{eq:def_L_inv_psi}).

The space $\rH_{f,p}^1(\Q,\ad^0\varrho_f)=\rH^1(\Q,\ad^0\varrho_f)$ is two-dimensional (see \cite[Lem.~3.2]{bellaiche-dimitrov}). We will now give a modular description of this space and use it to write down an explicit basis, under the assumption that 
\begin{equation}\label{generic-CM}
\sLm(\varphi) \cdot \sLm(\bar\varphi) \cdot (\sLm(\varphi) + \sLm(\bar\varphi)) \ne 0. 
\end{equation}
 The condition  \eqref{generic-CM} holds for  $\varphi$  quadratic,  whereas in all other cases the Weak $p$-adic 
 Schanuel Conjecture  predicts that $\sLm(\varphi)$ and $\sLm(\bar\varphi)$  are algebraically independent. 
 
 Fix    a square root   $\xi$ of  $\sLm(\bar\varphi) \sLm(\varphi)^{-1}\sS_{\bar\varphi}$.

It was conjectured in \cite{DLR4}  and proven in \cite{betina-dimitrov} that there is a natural isomorphism between $\rH^1(\Q,\ad^0\varrho_f)$ and the space   $S^\dagger_1\lsem f\rsem_0$ 
of genuine overconvergent modular forms belonging to the generalized eigenspace of $f$. Furthermore one computes in  \cite{betina-dimitrov-JTNB} the $q$-expansions of a basis $\{f_\sF^\dagger, f_\Theta^\dagger\}$ of the latter space.  
The corresponding basis $\{ \kappa_\sF, \kappa_\Theta\}$  of $\rH^1(\Q,\ad^0\varrho_f)$ is characterized by $\kappa_\sF(\tau)= \kappa_\Theta(\tau)=0$ and 
\[\kappa_{\Theta\vert \rG_K}= \begin{pmatrix} \eta_{\gp}-\eta_{\bar\gp} & 0 \\ 0 & \eta_{\bar\gp}-\eta_{\gp}\end{pmatrix}, 
 \text{ and  }  
\kappa_{\sF\vert \rG_K}= \begin{pmatrix} 0 &  \eta_\varphi \\   \bar\eta_\varphi  & 0 \end{pmatrix}.
\]
Here  $\eta_\varphi:\rG_K\to \bar\Q_p$ represents the cocycle $[\eta_\varphi]\in \rH^1(K,\varphi)$, 
normalized so that $\loc_{p}([\eta_\varphi])=\log_p\in \rH^1(\Q_p,\bar\Q_p)$, whereas  
$(\eta_\gp, \eta_{\bar\gp})$ is a  basis  of $\rH^1(K,\bar\Q_p)$ with $\eta_\gp$
is unramified outside $\gp$ such that $\loc_{p}(\eta_\gp)=\log_p$, and $\eta_{\bar\gp}=\eta_\gp\circ\tau$.

 \subsubsection{The $\sL$-invariant}  As in \textsection\ref{CM-case}  
 we parameterize the lines of $W$ by $s\in \mathbb{P}^1(\overline{\Q}_p)\setminus {\sS_\varphi}$ and $t\in \overline{\Q}_p$ 
  by letting  $W^+_{s,t}=\overline{\Q}_p(t w_1+w_2+s w_3)$ with the convention that 
 $W^+_{\infty,t}=\overline{\Q}_p (t w_1+w_3)$. The resulting projection  $W\to W^-=W/W^+_{s,t}$ is given by 
$\begin{pmatrix}a& b \\ c &-a\end{pmatrix}\mapsto \begin{pmatrix} a- t b \\ c - s b \end{pmatrix}$.

It is proven in  \cite[Lem.~1.4]{betina-dimitrov} that $\eta_{\bar\gp}(\Frob_\gp)=\eta_{\gp}(\Frob_{\bar\gp})=-\sL_\gp$. Using Local Class Field Theory,  the fact that $\eta_{\gp}$ is unramified outside $\gp$ allows us to compute:
\[ \eta_{\gp\vert K_\gp^\times}(p)=\frac{\eta_{\gp\vert K_\gp^\times}(u_\gp \bar{u}_\gp)}{\ord_p(u_\gp)}=
0+\frac{\log_p(\bar{u}_\gp)}{\ord_p(u_\gp)}=\sL_\gp.\]

Using also \cite[Prop~1.5]{betina-dimitrov}, one finds: 
\begin{align}  j_{D^\perp,f}(\kappa_\Theta)= \begin{pmatrix} 2\sL_\gp \\ 0 \end{pmatrix}, \qquad & 
j_{D^\perp,f}(\kappa_\sF)=  \begin{pmatrix}  -t \sL(\bar\varphi) \\ \sS_\varphi  \sL(\varphi)- s  \sL(\bar\varphi) \end{pmatrix},\\
j_{D^\perp,c}(\kappa_\Theta)= \begin{pmatrix} 1 \\ 0 \end{pmatrix}, \qquad& 
j_{D^\perp,c}(\kappa_\sF)=\begin{pmatrix} -t \\ \sS_\varphi - s  \end{pmatrix}.
\end{align}
As $e=2$, applying Proposition~\ref{dual-L-inv} yields
\begin{equation}\label{eq:irr-CM-L}
\begin{split}
&\sL(V,D_{s,t})=(-1)^2 \sL(W,W^+_{s,t})\\
&=\frac{\det \begin{pmatrix} 2\sL_\gp &  -t \sL(\bar\varphi) \\  0 & \sS_\varphi  \sL(\varphi)- s  \sL(\bar\varphi) \end{pmatrix}}
{\det  \begin{pmatrix} 1& -t \\ 0&\sS_\varphi - s  \end{pmatrix}}=
\frac{2\sL_\gp\cdot \left(\sS_\varphi  \sL(\varphi)- s  \sL(\bar\varphi)\right)}{\sS_\varphi - s}. 
\end{split}
\end{equation}

We would like to conclude by determining the regular submodules of arithmetic significance. 
As already noted $\ad^0(\varrho_f)$ being locally trivial at $p$, it does not point out in any particular direction. 
Nevertheless, according to \cite{betina-dimitrov} the form $f$ belongs to exactly $4$ Hida families, i.e. the representation 
$\varrho_f$ admits four non-trivial  deformations to  $\overline{\Q}_p\lsem X\rsem$, each determining an ordinary line. 
If this line is given by $\overline{\Q}_p(e_1+s e_2)$ then the line of the corresponding full flag in  $\ad^0(\varrho_f)$ 
is given by $\overline{\Q}_p\left(\begin{smallmatrix} -s & 1 \\ -s^2 & s \end{smallmatrix}\right)= 
\overline{\Q}_p( - s w_1 + w_2- s^2 w_3)$. 

When $s=0$ one finds $\sL(V,D_{0,0})=2\sL_\gp\cdot \sL(\varphi)$, corresponding to a $p$-adic variation  along the CM family $\Theta=\Ind_K^\Q(\psi\chi_\gp)$ containing $f$ whose ordinary line is  given by $W^+_{\Theta}=\overline{\Q}_p e_1$. Hence this 
$\sL$-invariant may be denoted by $\sL(\ad^0(\varrho_f), W^+_{\Theta})$.

Analogously, when $s=\infty$ one finds $\sL(V,D_{\infty,0})=2\sL_\gp\cdot \sL(\bar\varphi)$, corresponding to a $p$-adic variation  along the CM family $\bar\Theta=\Ind_K^\Q(\bar \psi\chi_\gp)$ containing $f$  whose ordinary line is  given by $W^+_{\bar\Theta}=\overline{\Q}_p e_2$. Hence this  $\sL$-invariant may be denoted by $\sL(\ad^0(\varrho_f), W^+_{\bar\Theta})$.

Finally, when   $s=\xi^{-1}$ corresponds to the ordinary line in one of the the non-CM families $\sF$ containing $f$, then $W_{\sF}^+=W^+_{-s^2,-s}$, and similarly $W_{\sF\otimes \varepsilon_K}^+=W^+_{-s^2,s}$.  
Note that these are regular submodules as $-\xi^{-2}=-\sLm(\varphi) \sLm(\bar\varphi)^{-1}\sS_{\varphi}\ne \sS_{\varphi}$ by \eqref{generic-CM}. Using formula \eqref{eq:irr-CM-L}, we get 
 \begin{align}\label{L-inv-F}
& \sL(\ad^0(\varrho_f), W_{\sF}^+)= \sL(\ad^0(\varrho_f), W_{\sF\otimes \varepsilon_K}^+)\\ \nonumber
&= \frac{2\sL_\gp\cdot \left(\sS_\varphi \sL(\varphi)+ \frac{ \sLm(\varphi)}{ \sLm(\bar\varphi)}\sS_{\varphi}\sL(\bar\varphi)\right)}{\sS_\varphi+ \frac{\sLm(\varphi)}{ \sLm(\bar\varphi)}\sS_{\varphi}} = 
\frac{2\sL_\gp\cdot \left(\sLm(\bar\varphi)  \sL(\varphi)+  \sLm(\varphi) \sL(\bar\varphi)\right)}{\sLm(\bar\varphi)+\sLm(\varphi)}. 
 \end{align}

Note that although $W_{\sF\otimes \varepsilon_K}^+$ differs from $W_{\sF}^+$,  the
 resulting  $\sL$-invariant is the same. 

\subsubsection{Compatibility with $p$-adic $L$-functions}
The most promising approach for the construction of $p$-adic $L$-functions for Artin motives is via $p$-adic deformation on an eigenvariety containing sufficiently many points of regular weight. This is explained by Bella\"{i}che \cite{eigenbook} in the case of weight one forms using the modular symbol construction of the Coleman--Mazur eigencurve. His construction of an adjoint 
$p$-adic $L$-function on the eigencurve is  very relevant to our context, except that it misses the cyclotomic variable. 
The upshot is that the Perrin-Riou $p$-adic $L$-function should be easier to construct  and arithmetically more significant to study for the regular submodules  $W^+_\Theta$, $W^+_{\bar \Theta}$, $W^+_\sF$ and $W^+_{\sF\otimes \varepsilon_K}$, as this would allow them to vary in  families. 

We now discuss  potential candidates for such  $p$-adic $L$-functions and the compatibility with our expression for the $\sL$-invariants.

First consider the CM families $\Theta$ and $\bar \Theta$. Note that $W^+_\Theta$ and $W^+_{\bar \Theta}$ actually lie in the first factor of the Galois decomposition $\ad^0(f)= \Ind_K^\Q \varphi \bigoplus \varepsilon_K$. Since Perrin-Riou's conjecture is compatible with the Artin formalism, a natural candidate for the $p$-adic $L$-function in the direction $W^+_\Theta$, resp. $W^+_{\bar \Theta}$, is 
\[L_p(\varphi,s) \cdot L_p(\varepsilon_K\cdot \omega_p,s), \quad \text{ resp. } \quad L_p(\bar{\varphi},s) \cdot L_p(\varepsilon_K\cdot \omega_p,s),\]
where $L_p(\varphi,s)$, resp. $L_p(\bar \varphi,s)$,  is the restriction to the cyclotomic line of Katz's $p$-adic $L$-function for $\varphi$, resp. for $\bar \varphi$ (e.g. see   \cite[\textsection4.3]{betina-dimitrov}) and $L_p(\varepsilon_K \cdot \omega_p,s)$ is the Kubota-Leopoldt $p$-adic $L$-function attached to $\varepsilon_K$. Note that Katz's construction indeed allows us to vary this $p$-adic $L$-function along the families $\Theta$ and  $\bar \Theta$. The expression for $\sL(\ad^0(\varrho_f),W^+_\Theta)$ also factors as a product of two $\sL$-invariants, namely, $\sL(\varphi)$ and $\sL(\varepsilon_K)= 2\cdot \sL_\gp$, the latter being  the $\sL$-invariant introduced by Ferrero and Greenberg \cite{ferrero-greenberg} in their proof of the Trivial Zero Conjecture for $L_p(\varepsilon_K\cdot \omega_p,s)$. The term $\sL(\varphi)$ is the $\sL$-invariant appearing in the Trivial Zero Conjecture for $L_p(\varphi,s)$ proposed in \cite[(45)]{betina-dimitrov}. When $\varphi$ is quadratic, this conjecture is deduced from Gross's factorization formula in \emph{loc. cit.}, end of \textsection4.3. A proof of this conjecture for general $\varphi$ (under mild hypotheses on $p$) based on the equivariant Tamagawa Number Conjecture proven by Bley \cite{bley} has been announced by B\"uy\"ukboduk and Sakamoto \cite{buyukboduk-sakamoto}. 
 Meanwhile Chida and Hsieh \cite{chida-hsieh} provided an entirely different proof removing those assumptions. 

Given a Hida family $\sG$ and letting $\sG^*$ be its dual family, Hida  \cite{hida-AIF88} attaches to the pair $(\sG,\sG^*)$ a three-variable $p$-adic meromorphic function $L_p(\sG,\sG^*,s)$ interpolating the special values of the Rankin-Selberg $L$-functions $L(g,h,s)$ of crystalline classical members $g\in \sG$, $h\in \sG^*$ of weights $\ell,m\geqslant 2$ at integers $m\leqslant s <\ell$. Let $L_p(g,h,s)$ be the specialization of $L_p(\sG,\sG^*,s)$ at $(g,h)$. When $g=h^*$ is a crystalline classical point of $\sG$ of weight $\ell=m\geqslant 2$, Dasgupta's factorization theorem \cite{dasgupta:factorization} asserts that $L_p(g,g^*,s)=\zeta_p(s-\ell+1)L_p(\ad^0(g),s)$ where $\zeta_p(s)$ is the $p$-adic zeta function and $L_p(\ad^0(g),s)$ is Hida-Schmidt's $p$-adic $L$-function \cite{schmidt} for the adjoint motive of $g$. It makes sense to apply the above construction to $\sG\in \{\Theta,\bar \Theta, \sF, \sF \otimes \varepsilon_K\}$, in which case the weight $1$ specialiazation of the $p$-adic meromorphic function $L_p(\sG,\sG^*,s)\cdot \zeta_p(s-\ell+1)^{-1}$ should be related to Perrin-Riou's $p$-adic $L$-function attached to $\ad^0(\varrho_f)$ and $W^+_{\sG}$. Some evidence of this fact for $\sG=\Theta,\bar \Theta$ lies in the analogy with \cite[Theorem 6.2]{rivero-rotger} which treats the case where $f$ is $p$-regular and the weight map is \'etale at $f$. Note that the simple (resp. double) trivial zero in the $p$-regular (resp. $p$-irregular) case is heuristically related to the smoothness (resp. non-Gorensteinness) of the eigencurve at $f$.

\bigskip
 { \noindent{\it Acknolwedgements:} { \small The authors are mostly indebted to D.~Benois for numerous discussions related to this project.
We would like also to thank A.~Betina and J.~Bella\"{i}che for helpful discussions. 
The research leading to this article is jointly funded by the Agence Nationale de Recherche ANR-18-CE40-0029 and the Fonds National de Recherche Luxembourg INTER/ANR/18/12589973
in the project {\it Galois representations, automorphic forms and their L-functions (GALF)}.
}

\bibliographystyle{siam}

\end{document}